\documentclass[a4paper,11pt]{amsart}

\usepackage[ascii]{inputenc}
\usepackage[ngerman,french,USenglish]{babel}
\usepackage{amscd}
\usepackage{hyperref}

\newcommand{\arxiv}[1]{\href{http://arxiv.org/abs/#1}{\nolinkurl{arXiv:#1}}}
\newcommand{\mr}[1]{\href{http://www.ams.org/mathscinet-getitem?mr=MR#1}{MR#1}}
\newcommand{\zbl}[1]{\href{http://www.zentralblatt-math.org/zmath/en/advanced/?q=an:#1&format=complete}{Zbl \nolinkurl{#1}}}
\newcommand{\urn}[1]{\href{http://nbn-resolving.de/urn:#1}{\nolinkurl{urn:#1}}}

\newtheorem{theorem}{Theorem}
\newtheorem{lemma}{Lemma}

\theoremstyle{definition}
\newtheorem{definition}{Definition}

\hypersetup{pdftitle={Chopped and sliced cones and representations of Kac-Moody algebras},pdfauthor={Thomas Bliem}}

\begin{document}

\title[Chopped and sliced cones]{Chopped and sliced cones and representations of Kac-Moody algebras}

\author{Thomas Bliem}

\address{\href{http://www.mi.uni-koeln.de/}{Mathematisches Institut}, Universit\"at zu K\"oln, Weyertal 86--90, 80931 K\"oln, Germany.}

\thanks{Supported by the Deutsche Forschungsgemeinschaft, \href{http://sfbtr12.uni-koeln.de/}{SFB/TR 12}.}

\begin{abstract}
We introduce the notion of a chopped and sliced cone in combinatorial geometry and prove two structure theorems for the number of integral points in the individual slices of such a cone. We observe that this notion applies to weight multiplicities of Kac-Moody algebras and Littlewood-Richardson coefficients of semisimple Lie algebras, where we obtain the corresponding results.
\end{abstract}

\maketitle

\section{Introduction}

The first occurrence of a polyhedral model of a representation is the notion of a Gelfand-Tsetlin pattern as considered by I. Gelfand and M. Tsetlin in 1950 \cite{gelfand1950}: Given the irreducible rational representation of $\mathit{GL}_n(\mathbf{C})$ of highest weight $\lambda = (\lambda_1 \geq \cdots \geq \lambda_n) \in \mathbf{Z}^n$, there is a basis of the representation space and a polytope $C^\lambda \subset \mathbf{R}^{n(n-1)/2}$ such that the basis vectors correspond bijectively to points in $C^\lambda$ with integral coordinates (or ``integral points'' for short). Moreover the family $(C^\lambda)$ depends linearly on $\lambda$ in the sense that it is given by linear inequalities on $\mathbf{R}^{n(n-1)/2} \times \mathbf{R}^n$. Moreover the constructed bases consist of weight vectors and the weight is given by a linear function on $\mathbf{R}^{n(n-1)/2}$. Hence all questions about dimensions of irreducible rational representations $V(\lambda)$ and of their weight spaces $V(\lambda)_\mu$ are reduced to the problem of determining the number of integral points in certain families of polytopes.

There are several generalizations of this; I would like to report on two of them. The first one is formed by the patterns introduced by P. Littelmann \cite{littelmann1998}: For these, $\mathit{GL}_n(\mathbf{C})$ is replaced by a symmetrizable Kac-Moody algebra $\mathfrak{g}$. To every highest weight representation there is an associated crystal graph as defined by M. Kashiwara and realized by P. Littelmann's path model. This colored graph is used to define, given any element $w \in W$ of the Weyl group with a fixed reduced decomposition, a family of polytopes $(C^\lambda)$ in $\mathbf{R}^{\mathrm{length}(w)}$ such that integral points in $C^\lambda$ correspond bijectively to elements of the crystal basis of the Demazure module $V_w(\lambda)$. Statements similar to the ones given for $\mathit{GL}_n(\mathbf{C})$ hold about the dependence of $C^\lambda$ on $\lambda$ and about the weight as a function on $C^\lambda$. In particular, if $\mathfrak{g}$ is finite dimensional one can choose $w$ to be the longest element of the Weyl group and integral points of $C^\lambda$ will correspond to elements of the crystal basis of the irreducible module $V(\lambda)$.

Note that Littelmann's patterns allow for dealing with problems concerning weight multiplicities in irreducible representations (as do the classical Gelfand-Tsetlin patterns). A. Berenstein and A. Zelevinsky \cite{berenstein2001} introduce a similar polyhedral model allowing for the expression of tensor product multiplicities $c^\mu_{\lambda\nu} = [V(\lambda) \otimes V(\nu) : V(\mu)]$ in terms of integral points in polytopes. As $\dim V(\lambda)_\mu = \lim_{\nu \to \infty} c^{\mu+\nu}_{\lambda\nu}$ this should be seen as a generalization of the polyhedral models for weight multiplicities.

The exploitation of the structure imposed on representations by the presence of these models has not been fully accomplished, even for the case of classical Gelfand-Tsetlin patterns. Interesting results in the case of classical Gelfand-Tsetlin patterns have been obtained by S. Billey, V. Guillemin and E. Rassart \cite{billey2004} as follows: They consider the equations and inequalities defining the polytopes $C^\lambda$ and the preimage $C^\lambda_\mu$ of a given weight $\mu$ under the (linear) weight map on $C^\lambda$. By a clever change of variables they obtain an expression of the weight multiplicity function for $\mathfrak{sl}_k(\mathbf{C})$ in the form
\[ \dim V(\lambda)_\mu = \Phi_{E_k} \bigl( B_k \bigl( \begin{smallmatrix} \lambda \\ \mu \end{smallmatrix} \bigr) \bigr), \]
where $E_k$ and $B_k$ are integral matrices, $\Phi_{E_k}$ is the vector partition function associated with $E_k$, and $\bigl( \begin{smallmatrix} \lambda \\ \mu \end{smallmatrix} \bigr) \in \mathbf{Z}^{2k}$ is the vector obtained by concatenating $\lambda$ and $\mu$ \cite[Th.\ 2.1]{billey2004}. This has some useful corollaries, for example using the structure theorem on vector partition functions \cite{blakley1964,dahmen1988,sturmfels1995} it follows that the weight multiplicity function is piecewise quasi-polynomial as a function of $(\lambda, \mu)$. This opens a number of questions, many of them answered in the paper, namely to count the number of regions of quasi-polynomiality (as a function of $k$), to compare them to the regions of polynomiality of the Duistermaat-Heckman measure, to determine the actual quasi-polynomials which constitute the weight multiplicity function, and to study their properties.

The purpose of this note is to set up a general framework, the chopped and sliced cone, in which the three above-mentioned cases \cite{gelfand1950,littelmann1998,berenstein2001} are subsumed.

\section{Chopped and sliced cones}
\label{CSCs}

We define a new structure in combinatorial geometry, the chopped and sliced cone.
To each chopped and sliced cone we associate a family of sequences of discrete measures.
We show that each such sequence converges weakly to an absolutely continuous measure.
Finally, to each chopped and sliced cone we associate a vector partition function such that important numerical quantities can be calculated by evaluating that vector partition function.

\subsection{Introduction}
\label{Def. of a CSC}

Let $K$, $\Lambda$ and $Q$ be free Abelian groups of finite rank.
Let $\tilde \Lambda_+$ and $R_+$ be free Abelian monoids of finite rank.
The free Abelian groups generated by a basis of $\tilde \Lambda_+$ respectively $R_+$ are denoted by $\tilde \Lambda$ and $R$.
On $\tilde \Lambda$, consider the partial order defined by $x \leq y$ if $y - x \in \tilde \Lambda_+$ and similarly for $R$.
Let $p : K \to \tilde\Lambda$, $q : K \to Q$, $r : K \to R$ and $s : \Lambda \to \tilde\Lambda$ be homomorphisms of Abelian groups.
Let $K_\mathbf{R} = K \otimes_\mathbf{Z} \mathbf{R}$ and similarly for the other groups.
The extensions of $p$, $q$, $r$ and $s$ to these real vector spaces are still denoted by the same symbol.
Figure \ref{Maps defining a CSC} contains an overview over these maps.

\begin{figure}
\[ \begin{CD}
R @<r<< K @>p>> \tilde\Lambda @<s<< \Lambda \\
@. @VVqV \\
@. Q
\end{CD} \]
\caption{Maps defining a chopped and sliced cone.}
\label{Maps defining a CSC}
\end{figure}

\begin{definition}
\label{def: chopped and sliced cone}
A 9-tuple $(K,$ $\Lambda,$ $Q,$ $\tilde \Lambda_+,$ $R_+,$ $p,$ $q,$ $r,$ $s)$ as above is called a \emph{chopped and sliced cone} if the set
\[ C^\lambda = \{ x \in K_\mathbf{R} : r(x) \geq 0,\ p(x) \leq s(\lambda) \} \]
is bounded for all $\lambda \in \Lambda$.
\end{definition}

If the extra structure is clear from the context, we may simply speak of the cone
\[ C = \{ x \in K_\mathbf{R} : r(x) \geq 0 \} \]
as being chopped and sliced.
For $\lambda \in \Lambda$ and $\beta \in Q$, define
\[ C^\lambda_\beta = \{ x \in C^\lambda : q(x) = \beta \} . \]

In order to simplify notation, for the rest of the chapter we assume without loss of generality that $\Lambda = \tilde\Lambda$ and $s = \mathrm{id}$.

\subsection{Measures associated with chopped and sliced cones}

For a fixed chopped and sliced cone $C$ we define some measures on $Q_\mathbf{R}$.
All appearing real vector spaces are topological spaces with the natural topology and as such equipped with the $\sigma$-algebra of Borel sets.
For any set $A$ let $\delta_A$ denote the counting measure associated with $A$.
If $A = \{ \beta \}$ contains only one element, we abbreviate $\delta_\beta = \delta_{\{\beta\}}$ for the corresponding Dirac measure concentrated in $\beta$.
For any measurable map $\phi$ let $\phi_*$ denote the push-forward of measures under $\phi$.
For $\lambda \in \Lambda$ and $n \in \mathbf{Z}_+$ let
\[ \mu_\lambda = \sum_{\beta \in Q} \lvert C^\lambda_\beta \cap K \rvert \cdot \delta_\beta = q_* \delta_{C^\lambda \cap K} \]
and
\[ \mu_\lambda^{(n)} = \frac{1}{n^{\mathrm{rk}(K)}} \bigl( \tfrac{1}{n} \mathrm{id}_{Q_\mathbf{R}} \bigr)_{\!*} \, \mu_{n\lambda} . \]
Let $\lambda_{K_\mathbf{R}}$ be the Lebesgue measure on $K_\mathbf{R}$, normalized such that $K$ has covolume $1$.
For a measurable subset $M \subset K_\mathbf{R}$ let $\lambda_M$ be  the $\lambda_{K_\mathbf{R}}$-absolutely continuous measure with density $1_M$, the characteristic function of $M$.
For $\lambda \in \Lambda$ we define
\[ \tilde\mu_\lambda = q_* \lambda_{C^\lambda}. \]
These measures are finite for any chopped and sliced cone $C$.

\begin{theorem} \label{th:schwache konvergenz}
Let $C$ be a chopped and sliced cone such that $q$ has full rank.
Let $\lambda \in \Lambda$.
Then $\mu_\lambda^{(n)}$ converges weakly towards $\tilde\mu_\lambda$ for $n \to \infty$. Moreover, the limit measure $\tilde\mu_\lambda$ is absolutely continuous with respect to Lebesgue measure on $Q_\mathbf{R}$.
\end{theorem}

For the proof, we provide the following fancy version of the definition of the Riemann integral:

\begin{lemma} \label{le:schwache konvergenz}
Let $M \subset K_\mathbf{R}$ be a bounded set such that the characteristic function $1_M$ of $M$ is Riemann integrable.
Then $\frac{1}{n^{\mathrm{rk}(K)}} \delta_{M \cap \frac{1}{n}K}$ converges weakly towards $\lambda_M$ for $n \to \infty$.
\end{lemma}

\begin{proof}
We denote the canonical pairing between functions and measures by $(\;\;{,}\;\;)$, i.e.\ $(f, \mu) = \int f \mu = \int f(x) \mu(dx)$.
Let $f \in C_b(K_\mathbf{R}, \mathbf{R})$ be an arbitrary continuous bounded function on $K_\mathbf{R}$.
We have to show that  $(f, \frac{1}{n^{\mathrm{rk}(K)}} \delta_{M \cap \frac{1}{n}K}) \to (f, \lambda_M)$ for $n \to \infty$.
For any $n \in \mathbf{Z}_{>0}$ we have
\[ \Bigl( f, \frac{1}{n^{\mathrm{rk}(K)}} \delta_{M \cap \frac{1}{n}K} \Bigr)
= \sum_{a \in M \cap \frac{1}{n}K} \frac{f(a)}{n^{\mathrm{rk}(K)}}
= \sum_{a \in \frac{1}{n}K} \frac{1_M(a) f(a)}{n^{\mathrm{rk}(K)}}. \]
As $1_M$ is Riemann integrable and $f$ is continuous, the product $1_M f$ is Riemann integrable and the above expression is a Riemann sum which converges for $n \to \infty$ towards $\int_{K_\mathbf{R}} 1_M(x) f(x) \, dx$.
As the Riemann integral and the Lebesgue integral coincide for Riemann integrable functions, we get
\begin{align*}
\Bigl( f, \frac{1}{n^{\mathrm{rk}(K)}} \delta_{M \cap \frac{1}{n}K} \Bigr)
&\to \int_{K_\mathbf{R}} 1_M(x) f(x) \, dx \\
&= \int_{K_\mathbf{R}} 1_M(x) f(x) \lambda_{K_\mathbf{R}}(dx) \\
&= \int_{K_\mathbf{R}} f(x) \lambda_M(dx) = (f, \lambda_M)
\end{align*}
for $n \to \infty$ as required.
\end{proof}

We also use the obvious fact that weak convergence and push-forward of measures commute:

\begin{lemma}
\label{le:schwache konvergenz des bildmasses}
Let $\phi : X \to Y$ be a continuous map between metric spaces and $(\mu_n)_{n \in \mathbf{Z}_{>0}}$ a sequence of finite measures on $X$ which converges weakly towards $\mu$.
Then the sequence $(\phi_* \mu_n)_{n \in \mathbf{Z}_{>0}}$ converges weakly towards $\phi_* \mu$.
\end{lemma}

Given these lemmas, we can directly verify Theorem \ref{th:schwache konvergenz}:

\begin{proof}[Proof of Theorem \ref{th:schwache konvergenz}]
By definition
\[ \mu_\lambda^{(n)}
= \frac{1}{n^{\mathrm{rk}(K)}} \bigl( \tfrac{1}{n} \mathrm{id}_{Q_\mathbf{R}} \bigr)_{\!*} \, \mu_{n\lambda}
= \frac{1}{n^{\mathrm{rk}(K)}} \bigl( \tfrac{1}{n} \mathrm{id}_{Q_\mathbf{R}} \bigr)_{\!*} \, q_* \delta_{C^{n\lambda} \cap K}. \]
As $q$ is linear
\[ \bigl( \tfrac{1}{n} \mathrm{id}_{Q_\mathbf{R}} \bigr)_{\!*} \, q_* = q_* \bigl( \tfrac{1}{n} \mathrm{id}_{K_\mathbf{R}} \bigr)_{\!*} \]
and furthermore
\[ \bigl( \tfrac{1}{n} \mathrm{id}_{K_\mathbf{R}} \bigr)_{\!*} \, \delta_{C^{n\lambda} \cap K}
= \delta_{\frac{1}{n} (C^{n\lambda}\cap K)}
= \delta_{C^{\lambda} \cap \frac{1}{n}K}, \]
so altogether
\[ \mu_\lambda^{(n)}
= q_* \Bigl( \frac{1}{n^{\mathrm{rk}(K)}} \delta_{C^\lambda \cap \frac{1}{n}K} \Bigr). \]
By Lemma \ref{le:schwache konvergenz}, $\frac{1}{n^{\mathrm{rk}(K)}} \delta_{C^\lambda \cap \frac{1}{n}K}$ converges weakly towards $\lambda_{C^\lambda}$ and by Lemma \ref{le:schwache konvergenz des bildmasses} we obtain the first statement.

By definition, $\lambda_{C^\lambda}$ is $\lambda_{K_\mathbf{R}}$-absolutely continuous on $K_\mathbf{R}$ with density $1_{C^\lambda}$, the characteristic function of $C^\lambda$.
We choose a section of the linear map $q : K_\mathbf{R} \to Q_\mathbf{R}$.
This yields an isomorphism $K_\mathbf{R} \cong Q_\mathbf{R} \oplus \mathrm{ker}(q)$ and all the fibers $q^{-1}(\{\beta\})$ of $q$ are identified canonically with $\mathrm{ker}(q)$.
Let the Lebesgue measure on $Q_\mathbf{R}$ be normalized such that $Q$ has covolume $1$.
Let the Lebesgue measure on $\mathrm{ker}(q)$ be normalized such that $\lambda_{K_\mathbf{R}}$ is the product measure of both.
By Fubini's theorem it follows that for any measurable set $A \subset Q_\mathbf{R}$ we have
\begin{align*}
\int_A \tilde\mu_\lambda
&= \int_{q^{-1}(A)} 1_{C^\lambda} \lambda_{K_\mathbf{R}} \\
&= \int_A \biggl( \int_{q^{-1}(\{\beta\})} 1_{C^\lambda_\beta} \lambda_{q^{-1}(\{\beta\})} \biggr) \lambda_{Q_\mathbf{R}}(d\beta) \\
&= \int_A \biggl( \int_{C^\lambda_\beta} \lambda_{q^{-1}(\{\beta\})} \biggr) \lambda_{Q_\mathbf{R}}(d\beta),
\end{align*}
i.e.\ $f_\lambda(\beta) = \int_{C^\lambda_\beta} \lambda_{q^{-1}(\{\beta\})}$ is the density of $\tilde\mu_\lambda$ with respect to $\lambda_{Q_\mathbf{R}}$.
\end{proof}

\subsection{Cones and vector partition functions}

We start by recalling some definitions:
Let $Y$ be a free Abelian group of finite rank and $X_+$ a free Abelian monoid of finite rank.
Let $X$ be the free Abelian group generated by a basis of $X_+$.
Let $E : X \to Y$ be a homomorphism such that $\ker E \cap X_+ = \{ 0 \}$.
One defines the \emph{vector partition function} $\Phi_E : Y \to \mathbf{Z}_+$ associated with $E$ by $\Phi_E(y) = \lvert \{ x \in X_+ : Ex = y \} \rvert$.
A function $f : Y \to \mathbf{Z}_+$ is called a \emph{quasi-polynomial} if there is a subgroup $Z \subset Y$ of finite index and a family $(f_{\bar y})_{y \in Y/Z}$ of polynomial functions on $Y$ such that $f(y) =  f_{\bar y}(y)$ for all $y \in Y$.
A \emph{fan} in $Y$ is a set $F$ of rational convex polyhedral cones in $Y_\mathbf{R}$ such that any face of a cone in $F$ is itself contained in $F$ and such that the intersection of any two cones in $F$ is a face of both.

\begin{theorem}
\label{th:kegel-als-vpf}
Let $C \subset K$ be a pointed chopped and sliced cone.
\begin{enumerate}
\item 
There are free Abelian groups $X$ and $Y$ and morphisms $E : X \to Y$ and $B : \Lambda \times Q \to Y$ such that $\lvert C^\lambda_\beta \cap K \rvert = \Phi_E \bigl( B \bigl( \begin{smallmatrix} \lambda \\ \beta \end{smallmatrix} \bigr) \bigr)$.
\item There is a fan $F$ in $\Lambda \times Q$ such that the function $\Lambda \times Q \to \mathbf{Z}_+$, $(\lambda, \beta) \mapsto \lvert C^\lambda_\beta \cap K \rvert$ is quasi-polynomial on any of the maximal cones of $F$ and vanishes outside $F$.
The fan $F$ and the quasi-polynomials associated with its maximal cones are effectively computable.
\end{enumerate}
\end{theorem}

For the proof, let $C$ be a pointed chopped and sliced cone.
We use the symbols $R$, $\Lambda$, $\tilde\Lambda$, $Q$, $p$, $q$, $r$, $s$ as in \S \ref{Def. of a CSC}.
As of the free Abelian group $K$, fix a free Abelian monoid $K_+$ of finite rank and let $K$ be the free Abelian group generated by a basis of $K_+$.
Because of the following lemma we can suppose that $C \subset K_+$.

\begin{lemma}
\label{lebla}
Let $C \subset \mathbf{R}^n$ be a rational pointed convex polyhedral cone. Then there is a matrix $A \in \mathit{GL}_n(\mathbf{Z})$ such that $AC \subset \mathbf{R}_+^n$.
\end{lemma}

\begin{proof}
Every pointed convex polyhedral cone is contained in a simplicial one, so there is an integral matrix $B \in \mathbf{Z}^{n \times n}$ such that $C \subset B\mathbf{R}_+^n$.
By the elementary divisors algorithm, there is a matrix $X \in \mathit{GL}_n(\mathbf{Z})$ and a permutation matrix $P$ such that $Y = XBP$ is upper triangular with positive components on the diagonal.
Write $Y = Y' D$, where $D$ is a positive diagonal matrix and $Y'$ is unipotent.
Choose a unipotent integral upper triangular matrix $Y''$ such that $Y'' \leq Y'$ (component wise).
Then $Y' \mathbf{R}_+^n \subset Y'' \mathbf{R}_+^n$.
Hence $C \subset B\mathbf{R}_+^n = X^{-1} Y P^{-1} \mathbf{R}_+^n = X^{-1} Y \mathbf{R}_+^n = X^{-1} Y' D \mathbf{R}_+^n = X^{-1} Y' \mathbf{R}_+^n \subset X^{-1} Y'' \mathbf{R}_+^n$.
So $A = Y''^{-1} X$ has the desired property.
\end{proof}

\begin{proof}[Proof of Theorem \ref{th:kegel-als-vpf}]
By Lemma \ref{lebla}, suppose that $C \subset K_+$.
Hence, the map $r : K \to R$ defining $C$ can be replaced by a map $\tilde r : K \to \tilde R$ to a free Abelian group of potentially lower rank by omitting the inequalities defining $K_+$.

For $\lambda \in \Lambda$, $\beta \in Q$ we have
\begin{align*}
\lvert C^\lambda_\beta \cap K \rvert
&=
\lvert \{ x \in K_+ : \tilde r(x) \geq 0,\ p(x) \leq s(\lambda),\ q(x) = \beta \} \rvert \\ \notag
&=
\begin{aligned}[t] \lvert \{ x \in K_+ : {} & \exists y \in \tilde R_+, z \in \tilde \Lambda_+ :
\tilde r(x) - y = 0, \\ & p(x) + z = s(\lambda),\ q(x) = \beta \} \rvert
\end{aligned} \\ \notag
&=
\Phi_E \begin{pmatrix} 0_{\tilde R} \\ s(\lambda) \\ \beta \end{pmatrix}
\end{align*}
for
\[ E =
\begin{pmatrix}
\tilde r & -\mathrm{id}_{\tilde R} & 0 \\
p & 0 & \mathrm{id}_{\tilde \Lambda} \\
q & 0 & 0
\end{pmatrix}
: K \times \tilde R \times \tilde \Lambda \to \tilde R \times \tilde \Lambda \times Q . \]
If we define
\[ \begin{pmatrix} 0_{\tilde R} \\ s(\lambda) \\ \beta \end{pmatrix} =
B \begin{pmatrix} \lambda \\ \beta \end{pmatrix} \]
for
\[ B = \begin{pmatrix} 0 & 0 \\ s & 0 \\ 0 & \mathrm{id}_Q \end{pmatrix}
: \Lambda \times Q \to \tilde R \times \tilde \Lambda \times Q \]
and $X = K \times \tilde R \times \tilde \Lambda$, $Y = \tilde R \times \tilde \Lambda \times Q$, we get part (1).

To show part (2), note that by Sturmfels' theorem \cite{sturmfels1995} there is a fan $\tilde F$ in $Y$ and a family $(g_{\mathfrak{\tilde c}})$ of quasi-polynomials, indexed by the maximal cones of $\tilde F$, such that $\Phi_E = g_{\mathfrak{\tilde c}}$ on $\mathfrak{\tilde c}$.
The fan $\tilde F$ and the quasi-polynomials $g_{\mathfrak{\tilde c}}$ are effectively computable, e.g.\ using an algorithm based on a residue formula \cite[Th.\ 3.1]{szenes2003} as in \cite{baldoni2006} or \cite{bliem2008}.
Let $F$ be the fan consisting of all $B^{-1}(\mathfrak{\tilde c})$ for $\mathfrak{\tilde c} \in \tilde F$.
For all $\mathfrak{c} \in F$ fix a $\mathfrak{\tilde c} \in \tilde F$ such that $\mathfrak{c} = B^{-1}(\mathfrak{\tilde c})$ and let $f_\mathfrak{c} = g_\mathfrak{\tilde c} \circ B$.
Then $F$ and $(f_\mathfrak{c})$ yield the theorem.
\end{proof}

\section{Application to weight multiplicities of Demazure modules}

For Kac-Moody algebras we use the notation as in V. Kac' book \cite{kac1990}, namely:
Let $A = (a_{ij})$ be an $(n \times n)$-generalized Cartan matrix with a realization $(\mathfrak{h}, \Pi, \Pi^\vee)$, i.e. $\Pi = \{ \alpha_1, \ldots, \alpha_n \} \subset \mathfrak{h}^*$, $\Pi^\vee = \{\alpha_1^\vee, \ldots, \alpha_n^\vee \} \subset \mathfrak{h}$ such that $\langle \alpha_i^\vee, \alpha_j \rangle = a_{ij}$. Let $Q \subset \mathfrak{h}^*$ be the root lattice and $P \subset \mathfrak{h}^*$ the weight ``lattice'' with dominant weights $P_+$.
Let $\mathfrak{h}' \subset \mathfrak{h}$ be the vector space generated by $\Pi^\vee$.
Let $\Lambda$ be the quotient of $P$ obtained by restricting linear forms to $\mathfrak{h}'$, and $\Lambda_+$ the image of $P_+$.
Then $Q$ and $\Lambda$ are free Abelian groups of rank $n$.
The restriction $P \to \Lambda$ is denoted by $\lambda \mapsto \bar\lambda$.

\begin{theorem}
\label{th3}
Let $\mathfrak{g}$ be a symmetrizable Kac-Moody algebra and $w \in W$ an element of length $l$ of its Weyl group. Let $K = \mathbf{Z}^l$.
\begin{enumerate}
\item There is a chopped and sliced cone $C \subset K_\mathbf{R}$ with $\Lambda$ and $Q$ as above, such that for all $\lambda \in P_+$  and $\beta \in Q$ we have
\[ \dim V_w(\lambda)_{\lambda-\beta} = \lvert C^{\bar\lambda}_\beta \cap K \rvert. \]
\item
There is an Abelian group $Y$ of finite rank, a homomorphism $B : \Lambda \times Q \to Y$ and a vector partition function $\Phi_E : Y \to \mathbf{Z}_+$ such that
\[ \dim V_w(\lambda)_{\lambda-\beta} = \Phi_E \bigl( B \bigl( \begin{smallmatrix} \bar\lambda \\ \beta \end{smallmatrix} \bigr) \bigr) \]
for all $\lambda \in P_+$, $\beta \in Q$.
\item The weight multiplicity function is piecewise quasi-polynomial.
Cones of quasi-polynomiality and the corresponding quasi-polynomials can be effectively computed.
\end{enumerate}
\end{theorem}

\begin{proof}
Because of Theorem \ref{th:kegel-als-vpf}, parts (2) and (3) follow from (1).
To show (1), in addition to $K$, $\Lambda$ and $Q$ as given above, we have to find $R$, $\tilde\Lambda$, $p$, $q$, $r$, and $s$ as in \S \ref{Def. of a CSC}, such that the stated equality holds.
We deduce this from the existence of P. Littelmann's patterns \cite{littelmann1998}.

Let $(B_w(\pi), (e_i)_{i=1}^n, (f_i)_{i=1}^n)$ be a path model for the crystal graph of $V_w(\lambda)$ as defined in \cite{littelmann1995}.
Fix a reduced decomposition $w = s_{i_1} \cdots s_{i_l}$ of $w$.
Then, for $\lambda \in P_+$, there is a subset $\mathcal{S}^\lambda \subset K$ such that $a \mapsto f_{i_1}^{a_1} \cdots f_{i_l}^{a_l} \pi$ defines a bijection $\mathcal{S}^\lambda \to B_w(\pi)$.
Let $\mathcal{S} = \bigcup_{\lambda\in P_+} \mathcal{S}^\lambda$.
By \cite[Pr.\ 1.5a and Co.\ 1]{littelmann1998}, $\mathcal{S}$ is the set of integral points in a rational convex polyhedral cone $C \subset K_\mathbf{R}$.
Let $R$ be a free Abelian group, the rank being the number of facets of $C$.
Then the inequalities defining $C$ can be stated in the form $r(a) \geq 0$ for a linear map $K_\mathbf{R} \to R_\mathbf{R}$.
As $C$ is rational, one can in fact choose $r : K \to R$.

By \cite[Pr.\ 1.5b]{littelmann1998}, $\mathcal{S}^\lambda$ is the set of integral points in a convex polytope $C^\lambda \subset C$, given as a subset of $C$ by the additional inequalities
\[
a_j + \sum_{k = j+1}^l \langle \alpha_{i_j}^\vee, \alpha_{i_k} \rangle a_k \leq \langle \alpha_{i_j}^\vee, \lambda \rangle
\]
for $j = 1, \ldots, l$.
These are $l$ inequalities, the left hand side of which depends linearly on $a$ and the right hand side of which depends linearly on $\bar\lambda$.
Hence we can choose $\tilde\Lambda = \mathbf{Z}^l$ and $p$, $q$ such that the inequalities can be written in the form $p(a) \leq s(\bar\lambda)$.

Given $a \in \mathcal{S}^\lambda$, the corresponding element $f_{i_1}^{a_1} \cdots f_{i_l}^{a_l} \pi$ has weight $\lambda - \sum_{j=1}^l a_j \alpha_{i_j}$.
So, for fixed $\beta \in Q$, the elements $a \in \mathcal{S}^\lambda$ of weight $\lambda - \beta$ are those satisfying $\sum_{j=0}^l a_j \alpha_{i_j} = \beta$.
The left hand side of this equation depends linearly on $a$, so if we denote this linear map by $q$, we have constructed a chopped and sliced cone having the property stated in the theorem.
\end{proof}

Consider the case where $w$ is the longest element of the Weyl group.
Then, for $\mathfrak{g} = \mathfrak{sl}_k(\mathbf{C})$, this can be proved using classical Gelfand-Tsetlin patterns \cite[Th.\ 2.1]{billey2004}.
For $\mathfrak{g} = \mathfrak{so}_5(\mathbf{C})$, I have explicitly computed the weight multiplicity function, as possible by part (3) of Theorem \ref{th3}, in \cite{bliem2008}.

\section{Application to Littlewood-Richardson coefficients}

In this chapter, we assume that $\mathfrak{g}$ is of finite type, i.e.\ a complex semisimple Lie algebra of rank $n$.
Let $Q \subset P \subset \mathfrak{h}^*$ denote the root lattice respectively the weight lattice.
For $\lambda, \nu \in P$, $\beta \in Q$ let $c^{\lambda\nu}_\beta = [V(\lambda) \otimes V(\nu) : V(\lambda+\nu-\beta)]$ be the multiplicity of $V(\lambda+\nu-\beta)$ in $V(\lambda) \otimes V(\nu)$.

\begin{theorem}
Let $\mathfrak{g}$ be a complex semisimple Lie algebra and $l = \mathrm{length}(w_0)$ the length of the longest element of the Weyl group.
\begin{enumerate}
\item For $\Lambda = P \times P$ and $K = \mathbf{Z}^l$ there is a chopped and sliced cone $C \subset K_\mathbf{R}$ such that for all $(\lambda, \nu) \in \Lambda$ and $\beta \in Q$ we have
\[ c^{\lambda\nu}_\beta = \lvert C^{(\lambda,\nu)}_\beta \cap K \rvert . \]
\item
There is an Abelian group $Y$ of finite rank, a homomorphism $B : P \times P \times Q \to Y$ and a vector partition function $\Phi_E : Y \to \mathbf{Z}_+$ such that
\[ c^{\lambda\nu}_\beta = \Phi_E \bigl( B \Bigl( \begin{smallmatrix} \lambda \\ \nu \\ \beta \end{smallmatrix} \Bigr) \bigr) \]
for all $\lambda, \nu \in P_+$, $\beta \in Q$.
\item The tensor product multiplicity function is piecewise quasi-polynomial.
Cones of quasi-polynomiality and the corresponding quasi-polynomials can be effectively computed.
\end{enumerate}
\end{theorem}

\begin{proof}
As before, (2) and (3) follow directly from (1), so we have to show (1).
This follows from \cite[Th.\ 2.4]{berenstein2001} as follows:
Fix a reduced decomposition $w = s_{i_1} \cdots s_{i_l}$ of $w$ and let $\mathbf{i} = (i_1, \ldots, i_l) \in \{ 1, \ldots, n \}^l$.
Let $^L\mathfrak{g}$ be the Langlands dual Lie algebra to $\mathfrak{g}$. For a finite-dimensional $^L\mathfrak{g}$-module $V$ and $\gamma, \delta \in P^\vee$ weights of $V$, a tuple $c = (c_1, \ldots, c_l) \in \mathbf{Z}_+^l$ is called an \emph{$\mathbf{i}$-trail} from $\gamma$ to $\delta$ in $V$ if $\sum_k c_k \alpha_{i_k} = \gamma - \delta$ and $e_{i_1}^{c_1} \cdots e_{i_l}^{c_l} : V_\delta \to V_\gamma $ is non-zero.
Then, by the above-cited theorem, $c^{\lambda\nu}_\beta$ is equal to the number of tuples $(t_1, \ldots t_l) \in \mathbf{Z}^l$ such that:

\begin{enumerate}
\item $\sum_{k=1}^l c_k t_k \geq 0$ for all $i \in \{1, \ldots, n \}$ and all $\mathbf{i}$-trails $c$ from $\omega_i^\vee$ to $w_0 s_i \omega_i^\vee$ in $V(\omega_i^\vee)$.
\item $\sum_{k=1}^l t_k \alpha_{i_k} = \beta$.
\item $\sum_{k=1}^l c_k t_k \geq - \langle \alpha_i^\vee, \lambda \rangle$ for all $i \in \{1, \ldots, n \}$ and all $\mathbf{i}$-trails $c$ from $\omega_i^\vee$ to $w_0 s_i \omega_i^\vee$ in $V(\omega_i^\vee)$.
\item $t_k + \sum_{j=k+1}^l a_{i_ki_j} t_j \leq \langle \alpha_{i_k}^\vee, \nu \rangle$ for all $k \in \{1, \ldots, l\}$.
\end{enumerate}

(1) is a system of linear inequalities, giving rise to a cone $C \subset K_\mathbf{R}$. (3) and (4) are systems of linear inequalities depending linearly on $\lambda$ respectively on $\nu$, so they define a chopping $C^{(\lambda,\nu)}$. (2) is a system of linear equations depending linearly on $\beta$, this defines the slices.
\end{proof}

\end{document}